\documentclass[12pt]{amsart}
\usepackage{amsmath, amscd, amssymb, amsthm,amsfonts,mathrsfs}

\usepackage[T1]{fontenc}
\makeatletter

\usepackage{color}
\usepackage{tikz}
\usetikzlibrary{matrix,arrows,positioning}
\tikzset{node distance=2cm, auto}

\ifx\pdfoutput\undefined
\usepackage{graphicx}
else
\fi

\DeclareMathOperator{\Sq}{Sq}

\DeclareMathOperator{\Inj}{Inj}

\DeclareMathOperator{\img}{im}

\DeclareMathOperator{\Endo}{End}
\DeclareMathOperator{\Ind}{Ind}
\DeclareMathOperator{\Res}{Res}
\DeclareMathOperator{\Fgt}{Forget}

\DeclareMathOperator{\Sym}{Sym}

\DeclareMathOperator{\End}{End}

\DeclareMathOperator{\Univ}{U}

\DeclareMathOperator{\K}{K}

\DeclareMathOperator{\Hom}{Hom}

\DeclareMathOperator{\Ext}{Ext}

\DeclareMathOperator{\NH}{NH}

\newcommand{\uli}[1]{\underline{\smash{\normaltext{#1}}}}
\newcommand{\conj}[1]{\quad\textnormal{ #1 }\quad}

\newcommand{\inp}[1]{\ensuremath{\langle #1 \rangle}}

\newcommand{\inpr}[1]{\left( \ensuremath{ #1 } \right)}

\newcommand{\module}[0]{\operatorname{-mod}}

\newcommand{\gmodule}[0]{\operatorname{-gmod}}
\newcommand{\gfmodule}[0]{\operatorname{-gfmod}}

\newcommand{\stabmodule}[0]{\operatorname{-\uli{mod}}}

\newcommand{\stabFgt}[0]{\operatorname{\uli{Forget}}}

\newcommand{\stabgmodule}[0]{\operatorname{-\uli{gmod}}}
\newcommand{\stabgfmodule}[0]{\operatorname{-\uli{gfmod}}}

\newcommand{\antip}[0]{S}

\newcommand{\normaltext}[1]{\textnormal{#1}}

\newcommand{\quantsl}[0]{\Univ_q\mathfrak{sl}(2)}

\newcommand{\cpquantsl}[0]{\aU^+}

\newcommand{\ZZ}{\mathbb{Z}}

\newcommand{\FF}{\mathbb{F}}
\newcommand{\CC}{\mathbb{C}}

\newcommand{\aA}{\mathcal{A}}

\newcommand{\aD}{\mathcal{D}}

\newcommand{\aM}{\mathcal{M}}

\newcommand{\aK}{\mathcal{K}}

\newcommand{\aP}{\mathcal{P}}

\newcommand{\aU}{\mathcal{U}}

\newcommand{\aQ}{\mathcal{Q}}

\newcommand{\mar}[0]{\aM}

\newcommand{\sap}[0]{\ensuremath{\aA_p}}

\newcommand{\ssap}[1]{\ensuremath{\aA_p(#1)}}

\newcommand{\dsap}[0]{\ensuremath{\aA^p}}

\newcommand{\dssap}[1]{\ensuremath{\aA^p(#1)}}

\makeatletter
\def\imod#1{\allowbreak\mkern2.5mu({\operator@font mod}\,#1)}
\makeatother

\renewcommand{\a}{\alpha}

\newcommand{\s}{\sigma}
\renewcommand{\d}{\delta}
\DeclareMathOperator{\Id}{Id}
\DeclareMathOperator{\Fl}{Fl}

\theoremstyle{plain}
\newtheorem{theorem}[subsection]{Theorem}
\newtheorem*{theoremun}{Theorem}

\newtheorem*{conjecture}{Conjecture}

\newtheorem{proposition}[subsection]{Proposition}
\newtheorem{corollary}[subsection]{Corollary}
\newtheorem{lemma}[subsection]{Lemma}
\theoremstyle{remark}

\newtheorem*{remark}{Remark}

\theoremstyle{definition}
\newtheorem{example}{Example}

\newtheorem{definition}[subsection]{Definition}

\newtheorem*{notation}{Notation}

\numberwithin{equation}{section}

\begin{document}
\title[Steenrod Structures on Quantum Groups]{Steenrod Structures on Categorified Quantum Groups}

\author[Beliakova and Cooper]{Anna Beliakova and Benjamin Cooper}
\address{Institut f\"{u}r Mathematik, Universit\"{a}t Z\"{u}rich, Winterthurerstrasse 190, CH-8057 Z\"{u}rich}
\email{anna\char 64 math.uzh.ch}
\email{benjamin.cooper\char 64 math.uzh.ch}

\begin{abstract}
Categorified quantum groups play an increasing role in quantum topology and
representation theory. The Steenrod algebra is a fundamental component of
algebraic topology. In this paper we show that categorified quantum groups
can be extended to module categories over the Steenrod algebra in a natural
way. This yields an intepretation of the small quantum group by Khovanov and
Qi.
\end{abstract}

\maketitle

\section{Introduction}\label{introsec}

Since their introduction quantum groups have played an important
role in quantum topology and representation theory. The language of
categorification has emerged as a mechanism for increasing the power of the
topological invariants associated to these objects. These new invariants
replace polynomials with homology groups which contain torsions and
symmetries that are not detectible the older setting.

In algebraic topology the Steenrod algebra $\sap$ is a powerful tool for
understanding $p$-torsion information. Each space $X$ determines a module
$H^*(X;\FF_p)$ over the Steenrod algebra and the Adam's spectral sequence
establishes a relationship between $\Ext$ groups of the module
$H^*(X;\FF_p)$ and the homotopy groups of $X$. As such the homological
algebra of modules over the Steenrod algebra has a long history, see
\cite{Steenrod, Margolis, Palmieri, Schwartz}.

The main goal of this paper is to introduce and study a Steenrod algebra
structure in the context of the categorifed quantum groups, see
\cite{KhLQG1, Rouquier}.  Our focus is on the categorification $\aU^+$ of
the positive half $U^+$ of $\quantsl$ and the 2-categories $\aU^+_{\FF_p}$
obtained from $\aU^+$ by taking coefficients in the finite fields $\FF_p$.

\begin{theoremun}
When coefficients are taken in the finite field $\FF_p$, the
categorification $\aU^+$ of the positive half $U^+$ of the quantum group
$\mathfrak{sl}(2)$ can be extended to a 2-category enriched in modules over
the Steenrod algebra.
\end{theoremun}

As a consequence, the category of modules over the Steenrod algebra acts on
the categorification
$$\sap\module \otimes\, \aU^+_{\FF_p} \to \aU^+_{\FF_p}.$$

Since this extension is obtained in a canonical way it is expected to apply
to many of the algebraic structures associated to categorified quantum
groups, such as knot homology theories.

This new structure is studied in Section \ref{expcompsec}. In particular, we
observe that the structure of a module category over the Steenrod algebra
gives rise to a natural family of differentials on $\aU^+_{\FF_p}$ called
Margolis differentials.

Khovanov and Qi have proposed a means by which categorified quantum groups
can be evaluated at a root of unity \cite{KhQi}. The authors prove that,
after passing to a finite field $\FF_p$, one can introduce a differential on
the 2-category $\aU^+_{\FF_p}$ which reduces it to a categorification of
$U^+$ in which the variable $q$ is set to a root of unity.  In Section
\ref{khqisec} we prove that the differential introduced in \cite{KhQi} is a
Margolis differential.

\begin{theoremun}
There exists a twisted Steenrod module structure on the nilHecke algebras
$\NH\otimes \FF_p$ which extends the $p$-differential graded structure
defining the categorification of the small quantum group.
\end{theoremun}

This result can be viewed as one explanation for the definitions appearing
in \cite{KhQi}. 

This paper is organized in the following way. In Section \ref{steenrodsec}
basic properties of the Steenrod algebra $\sap$ are reviewed. Section
\ref{pdgstr} contains a discussion of the relationship between
$p$-differential graded structures and Steenrod module algebra
structures. Section \ref{nilHeckesec} contains an introduction to the
nilHecke algebras which are used to define the categorification of $U^+$. In
Section \ref{modulestrsec} a standard Steenrod structure on the nilHecke
algebras is constructed and explored. Section \ref{qmodulestrsec} summarizes
the construction in \cite{KhHopfological, QiHopfological} and discusses some
of the consequences of the existence of Steenrod structures from this
perspective. In Section \ref{khqisec} we establish a relationship between
the Steenrod algebra and the small quantum groups of Khovanov and
Qi. Section \ref{proofsec} contains a number of proofs.

\section{Steenrod algebras}\label{steenrodsec}

In this section we define the Steenrod algebra $\sap$ and recall some of its
basic properties.  Full details can be found in the references, see
\cite{Steenrod, Margolis, Palmieri}.

Note that the algebras $\sap$ defined below are not the full Steenrod
algebras.  When $p$ is odd, we do not include the Bockstein $\beta$ and,
when $p$ is even, we set $P^n = \Sq^{2n}$. In this paper, $\beta$ and
$\Sq^{2n+1}$ will always be zero unless otherwise noted.  This is because
modules will only have graded elements of even degree.

A stable cohomology operation is a natural transformation of the cohomology
functor $H^*(-;R)$ which commutes with the suspension isomorphisms:
$$H^*(\Sigma X;R) \cong H^{*+1}(X;R).$$

When $R$ is the finite field $\FF_p$ of order $p$, stable cohomology
operations are called \emph{mod $p$ Steenrod operations}. There are 
basic operations,
$$P^n : H^*(X;\FF_p) \to H^{*+2n(p-1)}(X;\FF_p)\conj{where} n\geq 0,$$

called reduced $p$th powers. The axioms below suffice to characterize the
stable operations on cohomology rings $H^*(X;\FF_p)$.

\begin{definition}{(Steenrod axioms)}\label{steenrodaxioms}
\begin{enumerate}
\item $P^0 = \Id$.
\item If $|x| = 2n$ then $P^n x = x^p.$
\item If $2n > |x|$ then $P^n x = 0.$
\item The Cartan formula:
\begin{equation*}\label{cartanrels}
P^n(xy) = \sum_{i+j = n} P^i x \cdot P^j y.
\end{equation*}
\item The Adem relations: if $a,b>0$ and $a < p b$ then
\begin{equation*}\label{ademrels}
P^a P^b = \sum_{j = 0}^{[a/p]} (-1)^{a+j} {(p-1)(b-j) - 1 \choose a - pj} P^{a+b-j} P^j,
\end{equation*}
where the binomial coefficient is interpreted modulo $p$.
\end{enumerate}
\end{definition}

The example below follows from the axioms in Definition \ref{steenrodaxioms}
above.

\begin{example}
If $y\in H^2(X;\FF_p)$ is a cohomology class of degree 2 then 
\begin{equation}\label{powerofgen}
P^k y^n = {n \choose k} y^{n + k(p-1)} \conj{if} k \leq n
\end{equation}
and $P^k y^n = 0$ otherwise. 
\end{example}

The algebra $\sap$ is formed by grouping together all of the reduced $p$th
power operators.

\begin{definition}
The Steenrod algebra $\sap$ is the free $\FF_p$ algebra on the generators
$P^n$, $n \geq 0$, subject to the Adem relations $(5)$ above.
\end{definition}

By construction the algebra $\sap$ acts on the cohomology groups
$H^*(X;\FF_p)$ of any topological space $X$. However, it is important to
note that not every $\sap$ module comes from the cohomology of a space, see
Section \ref{unstabmsec}.

A grading on the algebra $\sap$ is determined by setting $|P^k| =
2k(p-1)$. There is a cocommutative coproduct, $\Delta : \sap \to
\sap\otimes\sap$, defined by,
$$\Delta(P^n) = \sum_{i+j = n} P^i \otimes P^j.$$

This choice is determined by the Cartan formula in Definition
\ref{steenrodaxioms} and makes $\sap$ into a Hopf algebra.  If $\Delta(x) =
\sum x_{(1)} \otimes x_{(2)}$ then the antipode $\antip : \sap\to\sap$ is
determined recursively by the relations
$$\antip(1) = 1\conj{and} \sum \antip(x_{(1)}) x_{(2)} = 0.$$

The dual Steenrod algebra $\dsap$ is the graded dual of the Steenrod algebra
$\sap$.  In degree $n$, $\dsap$ consists of $\FF_p$-valued functions on
degree $n$ elements of $\sap$. Since $\sap$ is a cocommutative Hopf algebra,
the dual algebra $\dsap$ is a commutative Hopf algebra. In
\cite{MilnorSteenrod}, Milnor proved that $\dsap$ is a polynomial ring,
$$\dsap \cong \FF_p[\xi_0, \xi_1, \xi_2, \ldots].$$

The first generator is $\xi_0 = 1$ and the $n$th generator $\xi_n$ is dual
to the element $P^{p^{n-1}}P^{p^{n-2}}\cdots P^p P^1\in\sap$. A grading on
$\dsap$ is determined by setting $|\xi_n| = 2(p^n - 1)$. The coproduct $\Delta :
\dsap\to\dsap\otimes\dsap$ is given by,
$$\Delta(\xi_n) = \sum_{i+j = n} \xi_i^{p^j}\otimes \xi_j.$$

\subsection{Margolis Differentials}\label{margolishomsec}
There is a family of special elements $d_t$ in the Steenrod algebra $\sap$
called Margolis differentials.

\begin{definition}\label{margolisdef} 
The $(s,t)$th Margolis differential $P^s_t$ is the element of $\sap$ which is dual to $\xi_t^{p^s}\in\dsap$.
\end{definition}

When $s<t$, one can show that $(P^s_t)^p = 0$ \cite{Margolis}. Primitive
Margolis differentials determine $p$-differentials on the cohomology rings
$H^*(X;\FF_p)$ of spaces $X$, see Section \ref{pdgstr}. These primitive
differentials $d_t = P^0_t$ can be defined recursively in terms of the
reduced power operations using the recurrence,
\begin{equation}\label{diffeqn}
d_1 = P^1\conj{ and } d_{i+1} = d_{i} P^{p^{i}} - P^{p^{i}} d_{i}.
\end{equation}

\begin{example}
The comodule structure on the polynomial ring $H^*(\CC P^{\infty};\FF_p) \cong
\FF_p[x]$, is given by the equation,
$$\varphi(x^{p^s}) = \sum_{k\geq 0} x^{p^{k+s}} \otimes \xi_k^{p^s},$$

which implies that Margolis differentials act according to the formula,
$$P^s_t (x^{p^i}) = \left\{\begin{array}{ll} x^{p^{k+s}} & i = s \\ 0 & i \ne s. \end{array} \right.$$
\end{example}

If $M$ is an $\sap$ module then the quotient
$\ker{(P^s_t)^{p-1}}/{\img(P^s_t)}$ is called \emph{Margolis homology}, see
\cite{AdamsMargolis}. This is akin to the slash homology considered in
\cite{KhQi}.

\subsection{Sub-Hopf Algebras}\label{subhopfsec}
In this section we recall the classification of sub-Hopf algebras of the
Steenrod algebra. This is the same as the classification of quotient Hopf
algebras of the dual Steenrod algebra $\dsap$.

\begin{theorem}{(\cite{AdamsMargolisSub})}\label{subhopfthm}
Every quotient Hopf algebra $B$ of the dual Steenrod algebra $\dsap$ is of the form,

\begin{equation}\label{quothopfeqn}
B = \dsap/(\xi_1^{p^{n_1}}, \xi_2^{p^{n_2}}, \xi_3^{p^{n_3}}, \ldots)
\end{equation}
where $\{n_i\}_{i=1}^{\infty}$ is a sequence of integers, $n_i \geq 0$, such
that for $0 < j < m$ either $n_m > n_{m-j} - j$ or $n_m \geq n_j$.
\end{theorem}

This theorem yields a large family of finite
dimensional sub-Hopf algebras $H$. For instance, $\sap$ is filtered:
\begin{equation}\label{filteq}
  \cdots \subset \ssap{n} \subset \ssap{n+1} \subset \cdots \subset \ssap{\infty} = \sap,
\end{equation}

where the algebras $\ssap{n}$ are dual to the algebras 
$$\dssap{n} = \dsap/(\xi_1^{p^n}, \xi_2^{p^{n-1}}, \xi_3^{p^{n-2}}, \ldots,
\xi_{n}^{p}, \xi_{n+1}, \xi_{n+2},\ldots).$$

The subalgebra $\ssap{n}$ is generated by $P^0 = 1$ and the first $n$
indecomposible Steenrod reduced $p$th powers, $\ssap{n} = \FF_p\inp{ P^{p^j} : 0\leq j < n}.$

If $d_t$ is a primitive Margolis differential from Section
\ref{margolishomsec} then
$$\mar_t = \FF_p\inp{d_t} \cong \FF_p[\partial]/(\partial^p)$$

is a sub-Hopf algebra. When $t = 1$, $d_1 = P^1$ and $\mar_1 = \ssap{1}$.

\subsection{$p$-DG structures}\label{pdgstr}

In this section we discuss the relationship between the $p$-differential
graded ($p$-DG) structures introduced in \cite{KhQi} and the Steenrod
algebra.

\begin{definition}{(\cite{KhQi} \S 2.2)}
Suppose that $k$ is a field of characteristic $p$ and $A$ is a $k$ algebra.
Then a $p$-DG structure on $A$ is a map $\partial : A \to A$ which satisfies the properties:
$$\partial^p = 0 \conj{ and } \partial(a\cdot b) = \partial(a)\cdot b + a \cdot \partial(b).$$
\end{definition}

Suppose that we choose the field $k$ to be $\FF_p$. Then a $p$-DG structure
on an algebra $A$ is the same as an $H$ module algebra structure on $A$ when
$H$ is the Hopf algebra $\FF_p[\partial]/(\partial^p)$.  We will refer to
any grading of $H$ as a $p$-DG structure.

In Sections \ref{margolishomsec} and \ref{subhopfsec} we saw that the
sub-Hopf algebra $\mar_t$ spanned by the Margolis differential $d_t$ is
isomorphic to $\FF_p[\partial]/(\partial^p)$. This suggests the following
two observations.

\begin{enumerate}
\item A Steenrod module structure on an algebra $A$ restricts to a natural
  family of $p$-differential graded structures given by the Margolis
  differentials described in Section \ref{margolishomsec}.

\item Conversely, identifying a given $p$-differential graded structure as
  the restriction of a Steenrod structure yields families of extensions
  along sub-Hopf algebras of the Steenrod algebra described in Section
  \ref{subhopfsec}.
\end{enumerate}

In Section \ref{modulestrsec} we show that the geometry underlying the
nilHecke algebras discussed in Section \ref{nilHeckesec} determines Steenrod
structures on these algebras in a natural way. In Section
\ref{qmodulestrsec} the homological algebra developed in the papers
\cite{KhHopfological, QiHopfological} is used to define analogues
$\aU^+_{\aA}$ of the categorified quantum group $\aU^+$ which have been
extended by this module structure.

An $\FF_2$ Steenrod structure on Khovanov homology was introduced by
Lipshitz and Sarkar, see \cite{LS}. The comparison above implies that these
results yield families of $p$-DG algebra structures on Khovanov homology.

In Section \ref{khqisec} the $p$-DG algebra structures on the nilHecke
algebras defined by Khovanov and Qi are interpreted homologically.

\section{nilHecke Algebras}\label{nilHeckesec}

In this section we recall the nilHecke algebras and some of their basic
properties.

\subsection{Algebraic Formulation}\label{algformsec}

\begin{definition}{(nilHecke algebra $\NH_n$)}
For each $n\geq 0$, the nilHecke algebra $\NH_n$ is the graded ring
generated by operators $x_i$ in degree $2$, $1\leq i \leq n$, and $\d_j$ in
degree $-2$, $1 \leq j < n$, subject to the relations:
\[
 \begin{array}{ll}
  \d_i^2 = 0,  & \d_i\d_{i+1}\d_i = \d_{i+1}\d_i\d_{i+1},\\
   x_i \d_i - \d_i x_{i+1} = 1, &  \d_i x_i - x_{i+1} \d_i = 1.
  \end{array}
\]

The operators also satisfy far commutativity relations,
\[
\begin{array}{ll}
\d_i x_j = x_j\d_i  \quad\text{if $|i-j|>1$}, &   \d_i\d_j = \d_j\d_i \quad \text{if $|i-j|>1$},\\   &  x_i x_j = x_j x_i \quad\text{for $1\leq i,j\leq n$}.
\end{array}
\]
\end{definition}

The generators $x_i$ are called polynomial generators and the generators
$\d_i$ are called divided difference operators.  There is a nice
diagrammatic presentation for these algebras in which a crossing is used to
depict $\d_i$ and a dot is used to denote $x_i$, see \cite{KhLQG1}.

The polynomial algebra $\aP_n$ on $n$ variables,
$$\aP_n = \FF_p[x_1,\ldots,x_n] \conj { where } |x_i| = 2,$$ 
serves as a defining representation for the nilHecke algebra $\NH_n$.
The symmetric group $\Sigma_n$ acts on $\aP_n$ by $\sigma(x_i) =
x_{\sigma(i)}$ for $\sigma \in \Sigma_n$. Let $\s_j$ denote the
transposition $(j,j+1)\in\Sigma_n$. The nilHecke algebra $\NH_n$ action on the polynomial algebra $\aP_n$ is defined on generators by letting $x_i$ act by multiplication and $\d_j$ act on $f\in\aP_n$ by the rule,
\begin{equation}\label{divopeq}
\d_j(f) = \frac{f-\s_j(f)}{x_j - x_{j+1}}.
\end{equation}

The divided difference operators act on products according to the formula
\begin{equation}\label{leibniz}
\d_i(fg) = \d_i(f) g + \s_i(f) \d_i(g).
\end{equation}

The ring of symmetric polynomials,
$$\Sym_n = \FF_p[x_1,\ldots,x_n]^{\Sigma_n} = \FF_p[e_1,\ldots,e_n] \conj{ where } |e_i| = 2i$$
is the ring of polynomials in $n$ variables which are invariant under the
action of the symmetric group. It is a polynomial algebra on the elementary
symmetric functions: $e_1,\ldots,e_n$. The subalgebra of polynomials invariant under the subgroup $\ZZ/2=\inp{\s_j} \subset \Sigma_n$ will be denoted by
$$\aP_n^{\s_j} = \{ f\in \aP_n : \s_j f = f \}\conj{ so that } \Sym_n\subset \aP^{\s_j}_n \subset \aP_n,$$

for each $n>0$ and for each $j=1,\ldots,n-1$. Equation \eqref{divopeq}
implies that $\d_i(e) = 0$ when $e\in \aP^{\s_j}_n$, $\d_i(f) \in
\aP^{\s_j}_n$ and $f\in \aP_n$ is any polynomial. Using \eqref{leibniz}
above this is equivalent to the $\Sym_n$-equivariance of each divided
difference operator:
$$\d_i(e f) = e \d_i (f) \conj{ where } e\in\Sym_n \normaltext{ and } f\in\aP_n.$$ 

The same is true for multiplication by $x_j$.  In fact, the nilHecke algebra
is the algebra of $\Sym_n$ linear operations on the ring $\aP_n$.
\begin{definition}{($\NH_n$)}\label{endnhdef}
$$\NH_n \cong \Endo_{\Sym_n}(\aP_n)$$
\end{definition}

See Section 2.5 \cite{KLMS} or  Proposition 3.5 \cite{Lauda}.

\subsection{Geometric Formulation}\label{geosec}

In this section we review one geometric interpretation for the algebraic
material in Section \ref{algformsec}. This is used to motivate the
construction in Section \ref{modulestrsec}. Standard references include
\cite{BGG, Demazure}, also the surveys \cite{Manivel, Hiller}.

A \emph{complete flag} is a sequence of nested spaces
$$ 0 =  V_0 \subset V_1 \subset \cdots \subset V_{n-1} \subset V_{n} = \CC^n \conj{ where } \dim_{\CC} V_i = i.$$

If $\Fl_n$ denotes the set of complete flags in $\CC^n$ then the stabilizer
of the induced $U(n)$ action on $\Fl_n$ is the $n$-torus $T \subset U(n)$
consisting of diagonal matrices. The identification,
$$\Fl_n \cong U(n)/T$$

endows $\Fl_n$ with the structure of a manifold. The cohomology
$H^*(\Fl_n;\FF_p)$ of each flag variety admits two descriptions. The Borel description, 
\begin{equation}\label{boreldesc}
H^*(\Fl_n;\FF_p)\cong \FF_p[x_1,\ldots, x_n]_{\Sigma_n} \cong \FF_p[x_1,\ldots,x_n]/\Sym_n^+
\end{equation}

where $\Sym_n^+$ consists of non-constant symmetric polynomials, is obtained
using Chern classes.  A different description can be given using a cellular
decomposition of the space of flags,
$$\Fl_n = \coprod_{w\in \Sigma_n} X_w.$$

If $l(w)$ is the length of the word $w\in\Sigma_n$ then duals $[X_w]\in H^{2
  l(w)}(\Fl_n;\FF_p)$ to the cycles determined by each cell form a basis for
the cohomology.

The nilHecke algebra arises when one attempts to relate these two
descriptions. Let $T_i$ be the subgroup of $U(n)$ associated to the Lie
algebra obtained by adjoining the $i$th root to the torus: if $\mathfrak{t}$ and $\mathfrak{t}_i$ denote the complexified Lie algebras of $T$ and $T_i$ respectively then
$$\mathfrak{t}_i = \mathfrak{g}_{\a_i} \oplus \mathfrak{t} \oplus \mathfrak{g}_{-\a_i}.$$
The bundles $\pi_i : U(n)/T \to U(n)/T_i$ determine the divided difference operators,
$$\d_i = {\pi_i}^* {\pi_i}_* : H^*(\Fl_n;\FF_p) \to H^{*-2}(\Fl_n;\FF_p)$$

which act by \eqref{divopeq} on the cohomology ring \eqref{boreldesc}.  The
relations satisfied by $\d_i$ imply that if $w\in S_n$ is expressed as a
reduced product of transpositions, $w = \s_{i_1} \s_{i_2} \cdots \s_{i_m},$
then operator $\delta_w = \d_{i_1} \d_{i_2} \cdots \d_{i_m}$ depends only on
the element $w\in \Sigma_n$. The theorem below uses the nilHecke algebra to
articulate the relationship between the two descriptions of
\eqref{boreldesc} given above.

\begin{theorem}\label{BGGD}
$$[X_w] = \delta_{w^{-1} w_0} x^\d$$ 
where $x^\d = x_1^{n-1} x_2^{n-2} \cdots x_{n-1}$ and $w_0\in\Sigma_n$ is the longest word in the symmetric group $\Sigma_n$. 
\end{theorem}

See \cite{BGG, Demazure}.

\subsubsection{Polynomial Rings}

The bundle $U(n) \to \Fl_n$ is classified by a map
$$f : \Fl_n \to BT.$$
The map $f^*$ which is induced by $f$ on cohomology is the quotient map from
the polynomial ring $\aP_n$ to its $\Sigma_n$-coinvariants,
$$\FF_p[x_1,\ldots, x_n] \to \FF_p[x_1,\ldots, x_n]_{\Sigma_n}.$$

The $T_i$ bundles $U(n)\to U(n)/T_i$ are classified by maps $\varphi_i : U(n)/T_i \to BT_i$ and there is a diagram
\begin{center}
\begin{tikzpicture}[scale=10, node distance=2.5cm]
\node (A1) {$\Fl_n$};
\node (B1) [right=1cm of A1] {$BT$};
\node (C1) [below=1cm of A1] {$U(n)/T_i$};
\node (D1) [below=1cm of B1] {$BT_i$};
\draw[->] (A1) to node {$$} (B1);
\draw[->] (A1) to node {$\pi_i$} (C1);
\draw[->] (B1) to node {$\varphi_i$} (D1);
\draw[->] (C1) to node {$$} (D1);
\end{tikzpicture}
\end{center}

from which it follows that the operators $\d_i = {\varphi_i}^*
{\varphi_i}_*$ determine the action of divided difference operators on the
polynomial ring $\aP_n$ discussed in Section \ref{algformsec}.

\section{Standard Steenrod Structures on nilHecke Algebras}\label{modulestrsec}

In this section we study an action of the Steenrod algebra $\sap$ on the
nilHecke algebras $\NH_n\otimes \FF_p$. The existence of this structure
derives from the geometric interpretation in Section \ref{geosec}. An
extension of the categorified quantum groups by this structure is found in
Section \ref{qmodulestrsec}. A non-standard Steenrod module structure on
$\NH_n\otimes \FF_p$ will be defined in Section \ref{khqisec}.

The map $i : T \to U(n)$ induces a map
\begin{equation}\label{inceq}
i^* : H^*(BU(n);\FF_p) \to H^*(BT;\FF_p)
\end{equation}

between cohomology rings of classifying spaces. This is the inclusion $i^* :
\Sym_n \to \aP_n$. Since they are cohomology rings, both rings in
\eqref{inceq} are $\sap$ module algebras and $i^*$ is a homomorphism of
$\sap$ module algebras. The map $i^*$ makes the polynomial algebra $\aP_n$
into a module over the algebra of symmetric polynomials $\Sym_n$.

The next proposition tells us that the collection of $\Sym_n$ equivariant
endomorphisms of $\aP_n$ is also a module over $\sap$.

\begin{proposition}\label{algstuffprop}
Suppose that $H$ is a commutative or cocommutative Hopf algebra. If $A$ is a
$H$ module algebra and $M,N$ are $H$ module left $A$ modules then the space
of maps $\Hom_A(M,N)$ is an $H$ module.
\end{proposition}

See Section \ref{algstuffproof} for a proof. If the objects above are graded
then the statement above remains true. The corollary below follows by
combining Proposition \ref{algstuffprop} and Definition \ref{endnhdef}. This
is the starting point for Section \ref{qmodulestrsec}. A different proof of
this corollary is provided by Proposition \ref{actiononnhprop}.

\begin{corollary}\label{nilHeckeprop}
For each prime $p$, the nilHecke algebra $\NH_n\otimes \FF_p$ with
coefficients in the field $\FF_p$ is a graded $\sap$ module algebra.
\end{corollary}
\begin{proof}
Recall from Definition \ref{endnhdef} that the nilHecke algebra $\NH_n$ is
the subalgebra of the endomorphism algebra of the polynomial ring $\aP_n$
consisting of $\Sym_n$-equivariant endomorphisms. After identifying the
polynomial ring $\aP_n$ with the cohomology ring $H^*(BT;\FF_p)$ and the
ring of symmetric polynomials $\Sym_n$ with the cohomology ring
$H^*(BU(n);\FF_p)$ both become module algebras over the Steenrod
algebra. The inclusion \eqref{inceq} makes $\aP_n$ a module over $\Sym_n$
and the result follows from Proposition \ref{algstuffprop} above by setting,
$A=\Sym_n$, $M=\aP_n$ and $N=\aP_n$.
\end{proof}

\begin{corollary}\label{criteriacor}
Suppose that $H$ is a commutative or cocommutative Hopf algebra over $k$ and
there is a map 
$$H\otimes k[x] \to k[x]$$ 

which determines an $H$ module algebra structure on the polynomial rings
$\aP_n = k[x]^{\otimes n}$.  If the inclusion $\Sym_n \hookrightarrow \aP_n$
is a map of $H$ module algebras then there is an induced $H$ module algebra
structure on the nilHecke algebras $\NH_n\otimes k$.
\end{corollary}

The corollary above follows from the preceding discussion.  It will be used
in the proof of Theorem \ref{khqidiffthm}.

\subsection{Explicit Computations}\label{expcompsec}

The cohomology ring $H^*(BU(n);\FF_p)$ is isomorphic to the $\FF_p$
algebra of symmetric polynomials in $n$ variables.
$$H^*(BU(n);\FF_p) \cong \Sym_n$$

It follows that if $P\in\sap$ is an element of the Steenrod algebra and $e
\in\Sym_n$ is a symmetric polynomial then $P e$ can be expressed as a
function of elementary symmetric polynomials.  For instance, when
$e_j\in\Sym_n$ is an elementary symmetric polynomial and $p=2$ the formula
for $P^n e_j$ is called the Wu formula,
\begin{equation}\label{wuformeqn}
P^n e_i = \sum_k {i - n \choose k} e_{n-k} e_{i+k}
\end{equation}
see \cite{Wu}.  For other primes $p$, this formula is more complicated, see
\cite{Peterson, Shay, Lance}.

Since the elementary symmetric polynomials $e_i \in H^*(BU(n);\FF_p)$
represent Chern classes, the Wu formulas are universal relations satisfied
by these invariants of complex vector bundles.

\begin{notation}
The following polynomial will appear in many of our statements and computations.
\begin{equation}\label{sdefeqn}
s_i = \d_i(P^1 x_i) = \d_i(x_i^p) = \frac{x_i^p - x_{i+1}^p}{x_i - x_{i+1}} = \sum_{k+l = p-1} x_i^k x_{i+1}^l.
\end{equation}

We will consistently use the symbol $s_i$ to denote this polynomial.
\end{notation}

The next proposition relates the Steenrod operations $P^d$ and the divided
difference operators $\d_i$. This equation holds among operators on the
polynomial ring $\aP_n$.

\begin{proposition}\label{commutatorprop}
If $d$ is a positive integer, then
$$P^d\d_i - \d_i P^d  =\sum^{d}_{j=1} (-1)^{j}s^j_i \d_i P^{d-j}$$
where $s_i = \d_i(P^1 x_i)$.
\end{proposition}

For a proof, see Section \ref{commutatorproof}.

Proposition \ref{commutatorprop} allows us to establish a formula for
the action of Steenrod reduced $p$th power operations on the polynomials $s_i
= \delta_i(P^1 x_i)$.

\begin{proposition}\label{sformulaprop}
For each $d\geq 0$,
\begin{equation*}
P^d s_i = \left\{
\begin{array}{lr} (-1)^d s_i^{d+1} & d < p \\ 0 & d \geq p \end{array}\right.
\end{equation*}
where $s_i = \delta(P^1 x_i)$.
\end{proposition}

For a proof of this proposition, see Section \ref{sformulaproof}.

\subsection{The nilHecke algebra $\NH_n$ as a module}\label{expmodformsec}

The action of the Steenrod algebra on polynomials $\aP_n$ is described by
Equation \eqref{powerofgen}. Since $\sap$ is a Hopf algebra, there is an
induced action on $\End(\aP_n)$. If we write the comultiplication and
antipode as
$$\Delta(P^n)= \sum P^n_{(1)}\otimes P^n_{(2)}=\sum^n_{i=0} P^i\otimes P^{n-i}\conj{and} \antip (P^d)=-P^d-\sum^{d-1}_{i=1} P^i \antip (P^{d-i})$$

repectively, then $P^n\in\sap$ acts on $f\in \End(\aP_n)$ by $f \mapsto \overline{P}^n f$ where
\begin{equation*}
(\overline{P}^n f)(y) = P^n_{(2)}f(\antip (P^n_{(1)})(y))\, 
\end{equation*}
for any polynomial $y\in\aP_n$. 

This formula agrees with the standard one in which $\antip$ is replaced by
$\antip^{-1}$. If $H$ is a commutative or cocommutative Hopf algebra then
the antipode $\antip$ satisfies $\antip = \antip^{-1}$, see Proposition 8.8
\cite{MilnorMoore}.

\begin{proposition}\label{actiononnhprop}
Let $P^n \in \sap$ be the $n$th Steenrod reduced power operation and $\d_i\in\NH_m$ a divided difference operator in the nilHecke algebra. Then 
$$\overline{P}^n \d_i = (-1)^n s_i^n \delta_i \conj{and} \overline{P}^n x_i = P^n x_i, $$
where $s_i = \delta_i (P^1 x_i)$. In particular, the action of $\aA_p$ on $\End(\aP_n)$ restricts to $\NH_n$.
\end{proposition}

For an argument, see Section \ref{actiononnhproof}.

\begin{remark}
This proposition can be used to compute the action of the Steenrod reduced
power operations on the Schubert polynomials $\delta_{w^{-1} w_0} x^\d$
mentioned in Theorem \ref{BGGD}.
\end{remark}

In Section \ref{pdgstr} we saw that every Steenrod algebra structure gives
rise to a family of $p$-differential graded algebra structures defined by
Margolis differentials. When coefficients are taken in the field $\FF_p$
Corollary \ref{nilHeckeprop} and Proposition \ref{actiononnhprop} state that
there is a standard Steenrod algebra structure on the nilHecke algebras. In
Theorem \ref{pdiffthm} below, Propositions \ref{sformulaprop} and
\ref{actiononnhprop} will be used to derive explicit formulas for these
Margolis differentials.

\begin{theorem}\label{pdiffthm}
There exists a standard family of $p$-differentials $\{d_k\}_{k=1}^{\infty}$
on the nilHecke algebras $\NH_n\otimes \FF_p$. Each differential $d_k$ is
uniquely determined by its values on generators:
$$d_k x_i = x_i^{p^k}\conj{ and } d_k\delta_i = (-1)^{l_n} s^{l_n}_i \delta_i.$$

where $l_n = p^n + p^{n-1} + \cdots + p + 1$.
\end{theorem}

In light of the discussion in Section \ref{margolishomsec}, it is only
necessary to verify the last equation. A proof is given in Section
\ref{pdiffproof}.

\section{Stable Module Structures on Categorified Quantum Groups}\label{qmodulestrsec}

In Sections \ref{nilHeckeprop} and \ref{khqisec}, Steenrod module structures
are defined on the nilHecke algebras.  In this section, we review how this
gives rise to extensions of the categorification $\cpquantsl$ of the
positive part of $\quantsl$. Material included here summarizes the
references, see \cite{KhHopfological, QiHopfological, KhQi} and
\cite{Margolis, Palmieri}.

In what follows, we fix a finite dimensional Hopf algebra $H$ and an $H$
module algebra $A$.

\subsection{Categories over Module Categories}\label{modmodcatsec}

Since $A$ is an object of the category $H\gmodule$ of graded $H$ modules, we
use the symbol $A_{|H}$ to denote $A$ endowed with the $H$ structure. There
is a category $A_{|H}\gmodule$ of left $A_{|H}$ module objects in the
category $H\gmodule$ and there is tensor product,
\begin{equation}\label{ot1}
\otimes : H\gmodule \times A_{|H}\gmodule \to A_{|H}\gmodule.
\end{equation}

If $M$ is an $H$ module and $N$ is a $A_{|H}$ module then $M\otimes N$ is an
object in the category $A_{|H}\gmodule$. It is an $H$ module because $N$ is
also an $H$ module and it is an $A_{|H}$ module because if $m\otimes n\in
M\otimes N$ and $a\in A$ then the rule
$$a\cdot (m\otimes n) = m \otimes (a\cdot n)$$

determines an $A_{|H}$ module structure on $M\otimes N$. Equation
\eqref{ot1} allows us to think of the category $A_{|H}\gmodule$ as a module
over the category $H\gmodule$.

When $\aA$ is a finite dimensional sub-Hopf algebra of the Steenrod algebra
$\sap$, Corollary \ref{nilHeckeprop} implies the nilHecke algebra $\NH_n$
can be endowed with an $\sap$ module structure. It follows that there is a
functor,
$$\otimes :  \aA\gmodule \times \NH_{n|\aA}\gmodule \to \NH_{n|\aA}\gmodule.$$

The relationship between these two abelian categories is quite interesting.
It may be best explored using other language, see \cite{Schwartz,
  Kerler}. However, the small quantum group in \cite{KhQi} cannot be
constructed without the relations introduced by a passage to the stable
category. This is the next step in our discussion.

\subsection{Stable and Derived Stable Categories}\label{stabcatsec}

\begin{definition}\label{stabmoddef}
For any finite dimensional Hopf algebra $H$, the category $H\gmodule$ has a
quotient $H\stabgmodule$, called the category of \emph{stable modules},
which is obtained by declaring a map $f : M\to N$ in $H\gmodule$ to be zero
when it factors through a projective $H$ module.
\end{definition}

One consequence of this definition is that two $H$ modules $M$ and $N$
become isomorphic in the stable category $H\stabgmodule$ if and only if there
exist projective modules $P$ and $Q$ so that 
\begin{equation}\label{stabeq}
M\oplus P \cong N \oplus Q
\end{equation}

in the category $H\gmodule$.  Since $H$ is a finite dimensional Hopf
algebra, $P$ and $Q$ can be taken to be free modules or injective modules.

If $A$ is an $H$ module algebra and $A_{|H}\gmodule$ is the associated category
of $A$ modules (see Section \ref{modmodcatsec}) then for any $K\in
A_{|H}\gmodule$, the module $H\otimes K$ is free \cite{Montgomery} and so
$$H\otimes K \cong 0,$$
in the stable category $H\stabgmodule$. This observation suggests the next definition.

\begin{definition}\label{momodef}
Let $H$ be a finite dimensional Hopf algebra. Then, for any $H$ module
algebra $A$, the category of \emph{stable $A$ modules} is given by the quotient
$$A_{|H}\stabgmodule = A_{|H}\gmodule/I,$$

where $I$ is the ideal of $A_{|H}$ module maps $f : M \to N$ which factor
through an $A_{|H}$ module of the form $H\otimes K$. The above category is
sometimes denoted in other ways, see \cite{QiHopfological} Sections 2.8 and
4.1.
\end{definition}

The quotient in Definition \ref{momodef} is compatible with Definition
\ref{stabmoddef} in the sense that the tensor product \eqref{ot1} descends
to a functor between stable categories.

Given an $A_{|H}$ module $M$ the forgetful functor determines an $H$ module $\Fgt(M)$. This induces a functor between stable categories,
$$\stabFgt : A_{|H}\stabgmodule \to H\stabgmodule.$$

\begin{definition}
A map $f : M \to N$ in the category $A_{|H}\stabgmodule$ is a
\emph{quasi-isomorphism} when the map $\stabFgt(f)$ is an isomorphism.
\end{definition}

Proposition 4 of \cite{KhHopfological} shows that quasi-isomorphisms $\aQ$
in $A_{|H}\stabgmodule$ form a localizing class; they can be inverted, see
\cite{GelfandManin} Section 3.2.

\begin{definition}
The \emph{derived category} of stable $A$ modules is the category obtained
from the stable category of $A$ modules by inverting quasi-isomorphisms.
$$\aD(A,H) = A_{|H}\stabgmodule[\aQ^{-1}]$$
\end{definition}

The tensor product descends to the quotient.
$$\otimes : \aD(k,H) \times \aD(A,H) \to \aD(A,H)$$

A theorem of Qi below shows that maps between algebras define induction
and restriction functors between derived categories.

\begin{theorem}\label{indres}
Suppose that $A$ and $B$ are $H$ module algebras. Then a map $f : A \to B$
determines induction and restriction functors,
$$\Ind_A^B : \aD(A,H) \rightleftarrows \aD(B,H) : \Res_A^B$$

which form an adjunction,
$$\Hom_{\aD(A,H)}(\Ind_A^B(M),N) \cong \Hom_{\aD(B,H)}(M,\Res_A^B(N)).$$

Moreover, when $f$ is a quasi-isomorphism the induction and restriction
functors define an equivalence of categories.
\end{theorem}

See \cite{QiHopfological} Section 8.

\subsection{Extending the Categorification}\label{extcat}
In this section we recall how to apply the material reviewed in Sections
\ref{modmodcatsec} and \ref{stabcatsec} to produce 2-categories
$\aU^+_{\aA}$ which are extensions the categorification $\aU^+_{\FF_p}$ of
$U^+$ by finite dimensional sub-Hopf algebras $\aA$ of the Steenrod algebra.

Fix a finite dimensional sub-Hopf algebra $\aA \subset \sap$ and a Steenrod
module structure on the nilHecke algebras $\NH_n$; either the standard one
in Corollary \ref{nilHeckeprop} or the one from Section \ref{khqisec}.
Since the nilHecke algebras $\NH_n$ are $\aA$ module algebras. There are
derived categories of the form $\aD(\NH_n, \aA)$. A proposition is needed to
relate these categories to one another.

\begin{proposition}\label{indresdefprop}
There is an $\aA$ module algebra homomorphism
$$i_{n+m}: \NH_n \otimes \NH_m \to \NH_{n+m}$$ 

which determines induction and restriction functors,
$$\Ind : \aD(\NH_n\otimes \NH_m,\aA) \rightleftarrows \aD(\NH_{n+m},\aA) : \Res$$
\end{proposition}
\begin{proof}
After labelling the generators $x_i, \d_i$ of $\NH_n$ by $i = 1,...,n$ and
the generators $x_j, \d_j$ of $\NH_m$ by $j = n+1,...,n+m$ the map $i_{n+m}$
is defined by, $a\otimes b \mapsto ab.$ The map $i_{n+m}$ is an algebra
homomorphism. It suffices to check to check that it is an $\sap$ module
map. This follows from the Cartan formula, see Definition
\ref{steenrodaxioms}.
\end{proof}

The 2-categories $\aU^+_{\aA}$ defined below extend the 2-category
$\aU^+_{\FF_p}$ which categorifies the positive half of the quantum group
$\quantsl$.

\begin{definition}{($\aU^+_{\aA}$)}\label{upluscat}
For each finite dimensional sub-Hopf algebra $\aA$ of the Steenrod algebra
there is a 2-category,
$$\aU^+_{\aA} = \bigoplus_{n} \aD(\NH_n,\aA)$$

with objects corresponding to the categories $\aD(\NH_n,\aA)$, morphisms are
generated by compositions of the induction and restriction functors defined
in Proposition \ref{indresdefprop} and 2-morphisms given by natural
transformations.
\end{definition}

The standard Steenrod structure explored in Section \ref{modulestrsec}
restricts to a countable family of $p$-DG structures on the nilHecke
algebras.  Choosing an appropriate subset of the differentials generates a
sub-Hopf algebra of the Steenrod algebra over which the categorification
persists.

In Section \ref{khqisec} the standard Steenrod structure is modified to
agree with choices made in the Khovanov and Qi construction \cite{KhQi}. A
similar picture may hold for this construction, see Section
\ref{khqitwistedsec}. 

It is not necessarily true that presentations of the nilHecke algebras
$\NH_n$ suffice to present the 2-category $\aU^+_{\aA}$. Since the derived
tensor product is used to define the induction and restriction functors,
there may be other natural transformations arising from the Bar
construction.

In light of the construction outlined above, the discussion found in Section
\ref{pdgstr} represents an appealing picture.

\subsection{Grothendieck Groups}\label{grothsec}

There is an analogue $\aU^{c+}_{\aA}$ of $\aU^+_{\aA}$ in Definition
\ref{upluscat} above defined by using derived categories of compact objects
$\aD^c(\NH_n,\aA)$ in place of the categories $\aD(\NH_n,\aA)$, see
\cite{QiHopfological}.  A direct analogue of Proposition 3.28 \cite{KhQi}
implies that the functors in Definition \ref{upluscat} above descend to
functors defined on derived categories of compact objects. The Grothendieck
group of these categories can be defined.

\begin{conjecture}\label{conjconj}
Suppose that $\aU^{c+}_{\aA}$ is the compact analogue of Definition
\ref{upluscat}, defined using the natural Steenrod structure in Corollary
\ref{nilHeckeprop}. Then any such extension decategorifies an extension of
the ground field,
$$\K_0(\aU^{c+}_{\aA}) \cong U^+ \otimes \K_0(\aD(k,\aA)),$$

and no two extensions are equivalent as categories,
$$\aA \not\cong \aA' \Rightarrow \aU_{\aA}\not\cong\aU_{\aA'}.$$
\end{conjecture}

\subsubsection{The Base Category}\label{habirosec}

In this section we determine the Grothendieck group of the category of
stable modules over a finite dimensional sub-Hopf algebra of the Steenrod
algebra.  These categories are the base categories over which the
constructions in Section \ref{qmodulestrsec} were performed.

The letters $H$ or $A$ will be used in statements which hold for any finite
dimensional Hopf algebra or algebra respectively. The letter $\aA$ will
represent a finite dimensional sub-Hopf algebra of the Steenrod algebra
$\sap$. 

The category of positively graded finite dimensional left $A$ modules is
denoted $A\gfmodule$. The translation functor $-[1]$ is an endomorphism
which makes the Grothendieck group $\K_0(A\gfmodule)$ a module over the ring
$\ZZ[q]$ where $q = \K_0(- [1])$.

\begin{lemma}\label{grothlemma}
Suppose that $\aA\gfmodule$ is the category of positively graded modules
over a finite dimensional sub-Hopf algebra $\aA$ of the Steenrod
algebra. Then the Grothendieck group $\K_0(\aA\gfmodule)$ is isomorphic to
the polynomial ring,
$$\K_0(\aA\gfmodule) \cong \ZZ[q].$$
\end{lemma}
\begin{proof}
Any such module $M = \oplus_{l>0} M_l,$ is filtered by setting, $M_{\geq t}
= \oplus_{l \geq t} M_l$.  The associated graded module $\oplus_t M_{\geq
  t}/M_{\geq t+1}$ is equal to $M$ in the Grothendieck group
$\K_0(\aA\gfmodule)$ and $\aA$ acts trivially on each summand for grading
reasons.
\end{proof}

The stable derived category $S(H\module)$ is the Verdier quotient
$\aK^\flat(H\module)/(\Inj H)$ of the homotopy category of chain complexes
of $H$ modules by the thick subcategory of complexes consisting of injective
modules. When $H$ is a finite dimensional Hopf algebra, Theorem 8.2
\cite{Krause} implies that the stable module category and the stable derived
category are equivalent as triangulated categories.
\begin{equation}\label{kraus}
H\stabmodule \cong S(H\module)
\end{equation}

Switching to finite dimensional graded modules and combining \eqref{kraus}
with Heller's Theorem \cite{WeibelK} implies that there is a short exact sequence,
\begin{equation}\label{hellerseq}
\K_0(\Inj H) \to \K_0(\aK^\flat(H\gfmodule))\to \K_0(H\stabgfmodule)\to 0.
\end{equation}

The Grothendieck group of the homotopy category of chain complexes is
isomorphic to the Grothendieck group of the underlying category. Since the
Hopf algebra $H$ is finite dimensional, injective modules are free modules
and there is an isomorphism,
\begin{equation}\label{injeq}
\K_0(\Inj H) \cong \inp{[H]}.
\end{equation}

\begin{corollary}
$$\K_0(\aA\stabgfmodule)\cong \ZZ[q]/\inpr{[\aA]},$$
where $[\aA] = \dim_q \aA$ is the image of $\aA$ in the Grothendieck group.
\end{corollary}
\begin{proof}
This follows by combining the Lemma \ref{grothlemma} with \eqref{hellerseq}
and \eqref{injeq}. Alternatively, one could argue this directly from the
definition of $\aA\stabgfmodule$, see \eqref{stabeq}.
\end{proof}

The sub-Hopf algebra $\aA$ of $\sap$ is dual to a quotient Hopf algebra
$\aA^\vee$ of $\dsap$. By Theorem \ref{subhopfthm}, $\aA^\vee$ is determined
by a sequence of integers, $n_1, n_2, n_3,\ldots,n_N$,
\begin{equation}\label{quotientalg}
\aA^{\vee}\cong \FF_p[\xi_1,\xi_2,\ldots,\xi_N]/\inpr{\xi_1^{n_1}, \xi_2^{n_2},\ldots,\xi_{N}^{n_N}}.
\end{equation}

As a vector space $\aA^{\vee}$ is isomorphic to a tensor
product of cyclic quotients of polynomial rings on one generator. Keeping
track of the grading yields the formula,
$$\dim_q \aA = \prod_{k=1}^{N} \frac{1-q^{n_k |\xi_k| }}{1-q^{|\xi_k|}}.$$

We have proven the following theorem.

\begin{theorem}\label{grothfinitethm}
Let $\aA$ be a finite dimensional sub-Hopf algebra of the Steenrod algebra $\sap$. Then the Grothendieck group of the category $\aA\stabgfmodule$ of positively graded stable modules is given by,

$$K_0(\aA\stabgfmodule)\cong \frac{\ZZ[q]}{ \inpr{\prod_{k=1}^{N} \frac{1-q^{n_k |\xi_k| }}{1-q^{|\xi_k|}}}},$$
for some sequence of integers $\{n_k\}_{k=1}^N$ satisfying the criteria of Theorem \ref{subhopfthm}.
\end{theorem}

The integers $n_k$ are always of the form $p^{r_k}$ and $|\xi_k| = 2(p^k
-1)$ is the degree of the generator $\xi_k$ in the dual Steenrod algebra.

The quotient above can be written in terms of cyclotomic polynomials,
\begin{equation}\label{cyclotomiceqn} 
\frac{1-q^{n_k |\xi_k| }}{1-q^{|\xi_k|}} = \prod_{\substack{d \mid n_k |\xi_k|\\d \nmid |\xi_k|}} \Phi_{d}(q)
\end{equation}

where $\Phi_d(q) \in\ZZ[q]$ is the $d$th cyclotomic polynomial.

\begin{remark}\label{gradingremark}
It is possible to regrade the Steenrod algebra $\sap$. For instance, setting
$|P^k| = 2k$ changes the grading of the dual Steenrod algebra $\dsap$ to one
determined by the rule $|\xi_k| = 2(p^k - 1)/(p-1)$. In this way, different
versions of Theorem \ref{grothfinitethm} can be produced.
\end{remark}

\begin{remark}
Suppose that $T\subset \Phi = \{\Phi_1(q),\Phi_2(q),\Phi_3(q),\ldots\}$.
The Habiro ring \cite{Habiro} associated with such a subset is the inverse limit,
$$\ZZ[q]^T = \varprojlim_{f\in\Pi(T)} \ZZ[q]/\inpr{f}, $$

where $\Pi(T)$ is the set multiplicatively generated by $T$.  It would be
interesting to develop a relationship between Habiro rings and the Steenrod
algebra by combining the categories $\aA\stabmodule$ in some fashion.
\end{remark}

\section{Relation to Small Quantum Groups}\label{khqisec}

The standard Steenrod structure from Section \ref{modulestrsec} does not
always restrict to the $p$-differential graded structure used by Khovanov
and Qi. In this section a different Steenrod structure is defined on the
nilHecke algebras. This module structure is shown to restrict to a
$p$-differential graded structure which agrees, up to sign, with the one
used to define the small quantum groups in \cite{KhQi}.

\subsection{Hopf Algebra Structures for the Small Quantum Group}\label{khqisubsec}

In this section we review some definitions from Khovanov and Qi \cite{KhQi}.

For each prime $p$, there is a Hopf algebra $H$,
\begin{equation}\label{khqihopfeqn}
H = \FF_p[\partial]/(\partial^p).
\end{equation}

The grading on $H$ is determined by $|\partial| = 2$, the coproduct is given
by $\Delta(\partial) = \partial\otimes 1 + 1\otimes \partial$ and the
antipode is $\antip(\partial) = -\partial$.

The polynomial algebra,
$$\aP_n = \FF_p[x_1,\ldots,x_n],$$

with grading $|x_i| = 2$ is given an $H$ module structure by choosing the $p$-differential $\partial$ which is determined by the rules
\begin{equation}\label{khqieqn}
\partial(x_i) = x_i^2\conj{and} \partial(xy) = \partial(x)y + x\partial(y).
\end{equation}

The algebra of symmetric polynomials, $\Sym_n$, inherits an $H$ module
structure which is determined by the equations
\begin{equation}\label{khqisym}
\partial(e_i) = e_1 e_i - (i+1) e_{i+1}
\end{equation}
when $i < n$ and $\partial(e_n) = e_1 e_n$ otherwise, where $e_i$ is the
$i$th elementary symmetric polynomial.

Together Definition \ref{endnhdef} and Proposition \ref{algstuffprop} imply
that $\NH_n$ becomes an algebra in the category $H\module$. The module
structure on $\NH_n$ is determined by its values on the generators.
$$\partial(x_i) = x_i^2\conj{ and } \partial(\d_i) = (x_i + x_{i+1})\d_i$$

\subsubsection{Twisted Structure}\label{khqitwistedsec}
Instead of considering the ring of polynomials $\aP_n$ as a module over $H$,
the ideal
\begin{equation}\label{khqitwistedmodeqn}
\aP_n(a) = \inp{x_2^{a} x_3^{2a}\cdots x_n^{(n-1)a}}\subset \aP_n \conj{ where } a\in \ZZ_+
\end{equation}

is used instead. The $H$ module structure on the polynomial ring $\aP_n$
induces an $H$ module structure on the module $\aP_n(a)$. The isomorphism
$\aP_n(a) \cong \aP_n$ implies that
$$\NH_n \cong \End_{\Sym_n}(\aP_n(a)).$$

The algebra $\NH_n$ does not change, but the $H$ module structure on $\NH_n$
does change. This deformed structure is determined by the equations,
$$\partial_a(x_i) = x_i^2\conj{and} \partial_a(\delta_i) = a + (a+1)x_i\delta_i + (a-1) x_{i+1}\delta_i.$$

\subsection{Interpretations of the Small Quantum Group}\label{geokhqisubsec}

In this section the construction of the standard Steenrod module structure
from Section \ref{modulestrsec} is modified to obtain a structure which
restricts, up to sign, to the $p$-differential graded structure found in
Section \ref{khqisubsec}.

The Hopf algebra $H$ in Section \ref{khqisubsec} is the same as the Hopf
algebra $\ssap{1} = \FF_p\inp{P^1}$ in Section \ref{subhopfsec}.

In Section \ref{modulestrsec}, the polynomial ring, $\aP_n$, inherits an $H$
module structure after it has been identified with the cohomology ring
$H^*(BT;\FF_p)$. This standard $H$ module structure is determined by the
equations,
\begin{equation}\label{psteenrodeqn}
P^1 y_i = y_i^{p} \conj{and} P^1(xy) = (P^1 x) y + x (P^1 y).
\end{equation}

When the prime $p = 2$ this agrees with Section \ref{khqisubsec}. For
instance, \eqref{khqisym} is a special case of the Wu formulas in Section
\ref{expcompsec}. For odd primes, the equations \eqref{psteenrodeqn} differ
from those of \eqref{khqieqn}.

In Section \ref{unstabmsec} we introduce a Steenrod structure which
restricts to one which agrees, up to sign, with \eqref{khqieqn} for all
primes. In Section \ref{ptorussec} this choice is interpreted
topologically. In Section \ref{twstrsec} the twisted structures in Section
\ref{khqitwistedsec} are addressed.

\subsubsection{A Stable Module Structure}\label{unstabmsec}

In this section we describe a Steenrod module structure on the nilHecke
algebras which restricts to the $p$-differential defined in Khovanov and
Qi. This algebraic construction is obtained from topological considerations
in Section \ref{twstrsec}.

\begin{definition}
There is a grading on the Steenrod algebra $\sap$ obtained by setting $|P^k|
= 2k$. The corresponding grading on the dual Steenrod algebra $\dsap$ is
determined by the rule $|\xi_k| = 2 (p^k -1)/(p-1)$.
\end{definition}

The Adem relations respect any grading of the form $|P^k| = Ck$ for
$C\in\ZZ$. The grading above is chosen to make the gradings compatible in
the theorem below.

\begin{theorem}\label{khqidiffthm}
Suppose that $\aP_n$ is the graded polynomial ring 
$$\aP_n = \FF_p[x_1,\ldots,x_n]$$

with $|x_i| = 2$. Then the equations
$$P^k x_i = \left\{\begin{array}{ll}   {p - 1 \choose k} x_i^{k+1} &  0\leq k < p\\ 0 & \normaltext{ otherwise } \end{array}\right.$$

determine a Steenrod module algebra structure on $\aP_n$.

Moreover, this structure induces a $\sap$ module algebra structure on the
nilHecke algebras $\NH_n\otimes \FF_p$. Up to sign, the $p$-DG structure
used by Khovanov and Qi agrees with a restriction of this induced structure.
\end{theorem}
\begin{proof}
This is a Steenrod structure because it a regrading of the structure defined
topologically in Section \ref{ptorussec}.  

Alternatively, Corollary \ref{criteriacor} implies that this choice induces
a structure on the nilHecke algebras when the map $i : \Sym_n
\hookrightarrow \aP_n$ is a homomorphism of $\sap$ module algebras. The
algebra $\Sym_n$ is a $\sap$ module algebra because the symmetric powers $
p_k = x_1^k + \cdots + x_n^k$ generate $\Sym_n$ and satisfy the equation,
$$P^d p_k = {n(p-1) \choose k } p_{d+k}.$$

Since $i\circ P^d = P^d\circ i$, $\aP_n$ is a left $\Sym_n$ module object in
the category of modules over the Steenrod algebra.  We conclude by observing
that, up to sign, the operation $P^1$ acts agrees with equation
\eqref{khqieqn}.
\end{proof}

In Definition \ref{steenrodaxioms} axioms for the action of the Steenrod
algebra $\sap$ on modules of the form $H^*(X;\FF_p)$ were given. It is not
the case that every $\sap$ module $M$ satisfies all of these axioms. Modules
which comes from cohomology rings are always unstable in the sense of the
following definition.

\begin{definition}
A module $M$ over the Steenrod algebra $\sap$ is \emph{unstable} when property
\begin{enumerate}
\setcounter{enumi}{2}
\item If $2n > |x|$ then $P^n x = 0$
\end{enumerate}

from Definition \ref{steenrodaxioms} holds. 
\end{definition}

The category of unstable modules over the Steenrod structure has been
studied extensively \cite{Schwartz}. The standard Steenrod algebra structure
is unstable, but there is no reason to demand that the choice
\eqref{khqieqn} extends to an unstable module structure. The structure
defined in Theorem \ref{khqidiffthm} is not unstable.

\subsubsection{$p$-Tori}\label{ptorussec}

Suppose that no change is made to the grading of the Steenrod algebra
$\sap$. If the equation $P^1 y = \pm y^2$ is to hold then the degree $|P^1|
= 2(p-1)$ implies $|y| = 2(p-1)$. There is a choice of space $X$ which has
cohomology ring $\FF_p[y]$ with $|y| = 2(p-1)$. Instead of using the spaces
found in Section \ref{geosec} we could attempt to use this space instead.

Consider the system of inclusions,
\begin{equation}\label{absyseqn}
\ZZ/(p) \to \ZZ/(p^2) \to \ZZ/(p^3) \to \cdots
\end{equation}

between cyclic groups. Taking the limit of this system gives a group,
$$SS^1_p = \ZZ/(p^{\infty}) = \varinjlim_{l} \ZZ/(p^l),$$

which we call the \emph{super $p$-circle}. The $\ZZ/(p)$ action on each
factor in equation \eqref{absyseqn} yields a $\ZZ/(p)$ action on
$SS^1_p$. Define the \emph{$p$-circle} by $S^1_p = SS^1_p/(\ZZ/(p))$. These
definitions are motivated by the following proposition.

\begin{proposition}
The cohomology of the classifying space of the super $p$-circle and its quotient by $\ZZ/(p)$ are given by,
$$H^*(BSS^1_p;\FF_p) \cong \Lambda(x) \otimes \FF_p[y] \conj{and} H^*(BS^1_p;\FF_p) \cong \FF_p[y],$$

where $|x| = 2p - 3$ and $|y| = 2(p-1)$.
\end{proposition}

If the polynomial algebra $\aP_n$ is defined to be
\begin{equation}\label{polynewlabel}
\aP_n = H^*((BS^1_p)^{\times n};\FF_p) = \FF_p[y_1,\ldots, y_n].
\end{equation}

Then restricting to the sub-Hopf algebra $\ssap{1}=\FF_p\inp{P^1}$ yields the
following equations,
\begin{equation}\label{thmeqn}
P^1 y_i = -y_i^{2} \conj{and} P^1(xy) = (P^1 x) y + x (P^1 y).
\end{equation}

This agrees with \eqref{khqieqn} up to sign. The Steenrod module structure
on $\aP_n$ is a regrading of the one described by Theorem \ref{khqidiffthm}.

\begin{remark}
One might replace $BS^1_p$ by $BSS^1_p$. Since elements can now have
odd degree, a study of this variation would require the full Steenrod
algebra.
\end{remark}

\subsubsection{Twisted Structures}\label{twstrsec}

In order to obtain the twisted formulas described in Section
\ref{khqitwistedsec} in the setting of Section \ref{unstabmsec} above, we
use the Thom space of the line bundle associated to the value of the
twisting parameter. An $n$-tuple $v = (t_1, \ldots, t_n)\in\ZZ^n$ determines
a representation $\chi : T^n\to \CC^{\times}$ by the assignment
$$(z_1, \ldots, z_n) \mapsto z_1^{t_1}\cdots z_n^{t_n}.$$ 

After extending $\chi$ to $\CC$, the Borel construction $X_\chi = ET^n
\times_{T^n} \CC $ gives a line bundle over $BT^n$. The Thom isomorphism
implies that the cohomology of the Thom space $H^*(X_\chi^{\CC})$ of this
line bundle is isomorphic to the ideal $\inp{z_1^{t_1}\cdots
  z_n^{t_n}}\subset H^*(BT^n)$.

For a fixed value of $a$, the formulas in Section \ref{khqitwistedsec} for
$P^1$ follow when the vector
$$(t_1,\ldots,t_n) = (0, a, 2a, \ldots, (n-1)a)$$

is used to determine the line bundle. The structure defined is quite
complicated. It would be interesting to explore this perspective further.

\section{Proofs}\label{proofsec}

In this section we prove some of the results which appear in earlier
sections. The arguments here are meant to be read in conjunction with prior
statements and discussion.

\subsection{Proof of Proposition \ref{algstuffprop}}\label{algstuffproof}

\newcommand{\id}{\mathrm{Id}}
\begin{proof}
Recall the Sweedler notation, 
$$\Delta^{(n)}(x)=\left(\Delta \otimes \id \right)\circ \Delta^{(n-1)}(x) = x_{(1)}\otimes \dots \otimes x_{(n+1)}\, . $$

in which the summation sign preceding the right hand side is dropped. 

In the statement of the proposition, $H$ is a Hopf algebra, $A$ is an
algebra in $H\module$ and $M,N$ are left $A$ modules in the category
$H\module$. Suppose that $h\in H$, $a\in A$, $m\in M$ and $f \in
\Hom_A(M,N)$. Then Proposition \ref{algstuffprop} is equivalent to the
identity,
$$(h\cdot f)(a\cdot m) = a (h\cdot f)(m).$$

The left hand side of this equation is equal to the first term below.
\begin{align}
h_{(2)} f(\antip(h_{(1)}) (a\cdot m) &=  h_{(2)} f((\antip(h_{(1)})_{(1)} a) \cdot (\antip(h_{(1)})_{(2)} m))\nonumber \\
&= h_{(2)} \cdot ( \antip(h_{(1)})_{(1)} a \cdot f(\antip(h_{(1)})_{(2)} m)\nonumber\\
\label{sweedler}
&= ({h_{(2)}}_{(1)} \antip(h_{(1)})_{(1)} a) \cdot ({h_{(2)}}_{(2)} f(\antip(h_{(1)})_{(2)} m))
\end{align}

Since $\antip(h_{(1)}) = \antip(h)_{(2)}$ and $\antip(h_{(2)}) =
\antip(h)_{(1)}$, equation \eqref{sweedler} becomes the equality below.
$$({h_{(3)}} \antip({h_{(2)}}) a) \cdot ({h_{(4)}} f(\antip({h_{(1)}}) m)) = a (h\cdot f)(m)$$

The last equality follows from the identity, $(\id \otimes S )\Delta = \id$.

\end{proof}

\subsection{Proof of Proposition \ref{commutatorprop}}\label{commutatorproof}
This first argument establishes a commutation relation between the action of
Steenrod reduced $p$th powers and the action of divided difference operators
on the polynomial ring.

\begin{proof}
Without loss of generality, we may consider $\d = \d_1$ in $\NH_2$ acting on
$\aP_2$. 
We claim that,
$$\d P^d - \sum_{i+j=d}(-1)^i s^i \d P^{j}\in\NH_2 $$ where $s=s_1$ is a
polynomial defined in the statement of the proposition.  By Definition
\ref{endnhdef} it suffices to show that the expression above is
$\Sym_2$-linear. The proof is by induction using the Cartan formula in
Definition \ref{steenrodaxioms} and the $\Sym_2$-linearity of $\d$.

If $e\in\Sym_2$ and $x\in \aP_2$ then
\begin{eqnarray*}
\lefteqn{P^d\d(ex) - \sum_{i+j=d}(-1)^i s^i \d P^{j}(ex) }\\
&&= \sum_{n+m = d} P^n (e)  P^m(\d(x)) - \sum_{i+n+m=d}(-1)^i s^i P^n(e) \d P^m(x) \\
&&= e [P^d, \d](x) + \sum_{\substack{n+m = d\\n>0}} P^n (e)  [P^m,\d](x) - \sum_{\substack{i+n+m=d\\i>0}} (-1)^i s^i P^n(e) \d P^m(x)\\
&&= e [P^d, \d](x) + \sum_{\substack{i+n+m = d\\ i,n>0}} (-1)^i s^i P^n (e) 
\d P^\beta (x) - \sum_{\substack{i+n+m=d\\i>0}}(-1)^i s^i P^n(e) \d P^m(x)\\
&& = e P^d \d(x) - e \sum_{i+j=d} (-1)^i s^i \d P^j(x) \, .
\end{eqnarray*}

Here we used that $P^ne \in \Sym_2$, this follows from the Wu formula in
Section \ref{expcompsec}.

Our claim implies that the equation,
$$ P^d \d  -\sum_{i+j=d} (-1)^i s^i \d P^{j} =r+t\d \,$$

holds for some unique choice of polynomials $r,t\in\aP_2$. Applying both
sides of the above equation to the polynomial $1$ shows that $r = 0$ and
acting on $x_1$ shows that $t=0$.
\end{proof}

\subsection{Proof of Proposition \ref{sformulaprop}}\label{sformulaproof}

Recall the special polynomial $s_i = \d_i(P^1 x_i)$ from
\eqref{sdefeqn}. The next proof establishes a formula for the action of
Steenrod reduced $p$th powers on $s_i$ for each $i> 0$.

\begin{proof}
If we omit subscripts and set $s = \d (P^1 x) = \d (x^p)$ then we wish to
show that for each $d\geq 0$,
\begin{equation*}
P^d s = \left\{
\begin{array}{lr} (-1)^d s^{d+1} & d < p \\ 0 & d \geq p \end{array}\right.
\end{equation*}
We begin by combining all of the Steenrod operations into a new operator
$$\hat{P} = \sum_{k \geq 0} P^k.$$

The proof consists of an application of Proposition \ref{commutatorprop} and
changing the order of summation.
\begin{align*}
\hat{P}s &= \sum_{k\geq 0} P^k \d (x^p) = \sum_{k\geq 0} \sum_{i=0}^k (-1)^{k-i} s^{k-i} \d(P^i x^p)\\
&= \sum_{i\geq 0} \d (P^i x^p) \sum_{k\geq i} (-1)^{k-i} s^{k-i}.
\end{align*}

So that $\hat{P} s = \d (\hat{P} x^p)/(1+s)$. The Cartan relation in
Definition \ref{steenrodaxioms} implies that
$$\hat{P}(x^p) = (\hat{P} x)^p = (x + x^p)^p = x^p + x^{p^2},$$

and $\d(x^p + x^{p^2}) = s + s^{p+1}$ which allows us to conclude that,
$$\hat{P} s = s\frac{1+s^p}{1+s} = \sum_{k=0}^{p-1} (-1)^k s^{k+1}.$$

Examining the graded components of each expression establishes Proposition
\ref{sformulaprop}.
\end{proof}

\subsection{Proof of Proposition \ref{actiononnhprop}}\label{actiononnhproof}
In this section we justify the formulas for the action of the Steenrod
algebra $\sap$ on the nilHecke algebras $\NH_n$.  Recall that an element
$P\in\sap$ in the Steenrod algebra acts on $f\in \End(\aP_n)$ by $f \mapsto
\overline{P}^n f$ where
\begin{equation}\label{bar}
(\overline{P}^n f)(y) = P^n_{(2)}f(\antip(P^n_{(1)})(y))\, 
\end{equation}
for any polynomial $y\in\aP_n$.

\begin{proof}
Let us compute the action of $P^n\in\sap$ on $\d_i\in\NH_m$. We set
$\d=\d_i$ and $s = s_i$ where $s_i=\d_i(P^1 x_i)$. Equation \eqref{bar}
implies that
$$\overline{P}^n(\d) = P^n \d +\sum^{n-1}_{j=1} P^j \d \antip(P^{n-j}) -\d(\antip(P^n)) =  [P^n,\d] + \sum^{n-1}_{j=1} [P^j,\d] \antip(P^{n-j})\, .$$

Proposition \ref{commutatorprop} above yields the equation
\begin{align*}
\bar P^n(\d)&= \sum^{n}_{j=1} (-1)^j s^j \d P^{n-j} 
+ \sum^{n-1}_{i=1} \sum^{i}_{j=1} (-1)^j s^j \d P^{i-j} \antip(P^{n-i}).
\end{align*}
Separating the terms with $i=j$ in the second summation we obtain,
\begin{align*}
\overline{P}^n(\d) &= \sum^{n}_{j=1} (-1)^j s^j \d P^{n-j} 
+ \sum^{n-1}_{j=1} (-1)^j s^j \d \antip(P^{n-j})
+\sum^{n-1}_{i=2} \sum^{i-1}_{j=1} (-1)^j s^j \d P^{i-j} \antip(P^{n-i})\\&
=(-1)^n s^n \d - 
\sum^{n-1}_{j=1} \sum^{n-j}_{i=1} (-1)^j s^j \d P^{i} \antip(P^{n-i-j})
+\sum^{n-1}_{i=2} \sum^{i-1}_{j=1} (-1)^j s^j \d P^{i-j} \antip(P^{n-i})\, .
\end{align*}
After a change of variables the last two sums cancel implying the desired
result. The action of $P^n\in\sap$ on $x_i\in\NH_m$ is left as an exercise
to the reader.
\end{proof}

\subsection{Proof of Theorem \ref{pdiffthm}}\label{pdiffproof}

The proof of the theorem will hinge upon formula from Euler's oeuvre \cite{Euler},
\begin{equation}\label{eulereq}
\left( \frac{1-x^p}{1-x} \right)^n = (1 + x + \cdots + x^{p-1})^n = \sum_{k=0}^{np} {n \choose k}^p x^k,
\end{equation}

where the symbol ${n \choose k}^p$ is equal to the usual binomial coefficient when $p=2$ and
$${n \choose \lambda}^p = \sum_{k=0}^{\lambda} {n \choose \lambda - k}^2 {\lambda - k \choose k}^{p-1}$$

when $p> 2$.

\begin{lemma}\label{generalizedpslemma}
$$P^k s^n = \left\{\begin{array}{ll} {n \choose k}^p (-1)^k s^{k+n} & k \leq np \\ 0 & k > np. \end{array}\right.$$
\end{lemma}
\begin{proof}
If $\hat{P} = \sum_{k\geq 0} P^k$  then Lemma \ref{sformulaprop} implies that
$$\hat{P} s = \sum_k P^k s = \sum_{k = 0}^{p-1} (-1)^k s^{k+1} = s \frac{1-(-s)^p}{1-(-s)}.$$

The Cartan formula implies that
$$\hat{P} s^n = (\hat{P} s)^n = s^n \left( \frac{1-(-s)^p}{1-(-s)} \right)^n,$$

and using Euler's formula above yields the statement of the lemma.
\end{proof}

We now use Lemma \ref{generalizedpslemma} to compute what the primitive
Margolis differentials $d_k$ do to the divided difference operators
$\delta_i\in\NH_n$.

\begin{proof}
Let $\delta = \delta_i$ and $s = s_i$. The differentials are defined
recursively by the formula,
\begin{equation}\label{marggg}
d_1 = P^1 \conj{ and } d_{n+1} = [d_n, P^{p^n}].
\end{equation}

Proposition \ref{actiononnhprop} implies that
$$d_1 \delta = -s \delta.$$

It follows from the recursion \eqref{marggg} and induction that,
$$d_{n+1}\delta = C_n (-1)^{l_n} s^{l_n} \delta,$$

where $l_n = (p^{n+1}-1)/(p-1) = p^n + p^{n-1} + \cdots + p + 1$ and
\begin{equation}\label{cneq}
C_n = C_{n-1} - C_{n-1} \sum_{i=0}^{p^n} {l_{n-1} \choose i}^p.
\end{equation}

The sum above is equal to zero because setting $x = 1$ in \eqref{eulereq} implies that
$$\sum_{i=0}^{n(p-1)} {n \choose i}^p \equiv 0 \imod{p}.$$

If $n = l_{n-1}$  then the sum becomes
$$\sum_{i=0}^{p^n - 1} {l_{n-1} \choose i}^p \equiv 0 \imod{p}$$

Since ${l_{n-1} \choose p^n}^p$ is zero the sum in \eqref{cneq} is zero and $C_n = 1$ for all $n> 0$.
\end{proof}

{\bf Acknowledgements.}  The authors would like thank Nicholas Kuhn for some helpful conversations.

B. Cooper was supported in part by the University of Virginia and the Max
Planck Institute for Mathematics.

\bibliographystyle{alpha}  
\bibliography{steenrod}  

\end{document}